\documentclass[psamsfonts]{amsart}

\usepackage{amssymb,amsfonts}
\usepackage[all,arc]{xy}
\usepackage{enumerate}
\usepackage{mathrsfs}

\newtheorem{theorem}{Theorem}[section]
\newtheorem{corollary}[theorem]{Corollary}

\newtheorem{lemma}[theorem]{Lemma}

\theoremstyle{definition}
\newtheorem{definition}[theorem]{Definition}

\newtheorem{notation}[theorem]{Notation}

\theoremstyle{remark}

\makeatletter
\let\c@equation\c@theorem
\makeatother
\numberwithin{equation}{section}

\title{SOME REMARKS ON QUASINEARLY SUBHARMONIC FUNCTIONS}

\author{MANSOUR KALANTAR}

 \email{kalantarm@smccd.edu} 
 
\date{}

\begin{document}

\begin{abstract}
We prove some basic properties of quasi-nearly subharmonic functions and quasi-nearly subharmonic functions in the narrow sense.
\end{abstract}

\maketitle

\section{NOTATION, DEFINITIONS AND PRELIMINARIES} 
\begin{notation}
In what follows $ D $ is a domain of $ \mathbb{R}^{N} $ $ (N\geq2) $. The ball of center $ x\in D $ and radius $ r>0 $ is noted $ B(x,r) $. We write $ \nu_{N} $ for the volume of the unit ball, and $ \lambda $ designates the $ N- $dimensional Lebesgue measure.
\end{notation}

\begin{definition}

A function $ u:D\rightarrow [-\infty,+\infty) $ is called nearly subharmonic, if $ u $ is (Lebesgue) measurable and satisfies the mean value inequality; i.e, for all ball $ B(x,r) $ relatively compact in $ D $, 
$$ u(x)\leq \frac{1}{\nu_{N}r^{N}}\int_{B}(x,r)u(\xi)d\lambda(\xi). $$
\end{definition}
This is a generalization of subharmonic functions, in the sense of, and  given by J. Riihentaus that differs slightly from the standard definition of nearly subharmonic functions (see \cite{Rii2} and the references therein). 
\begin{theorem}
 A function $ u:D\rightarrow [-\infty,+\infty) $ is nearly subharmonic if and only if there exists a subharmonic function  that is equal to $ u $ almost everywhere in $ D $. Further, if such a function exists, it is unique and is given by the upper-semi regularization of $ u $:
 $$ u^{*}(x)=\limsup_{\zeta\rightarrow x}u(x). $$ 
\end{theorem}
See \cite[pg 14]{Herv}

\begin{definition}
A function $ u:D\rightarrow [-\infty,+\infty) $ is called $ K- $quasi-nearly subharmonic, if $ u $ is (Lebesgue) mearsurable, its positive part $ u^{+} $ is locally integrable and there exists a constant $ K=K(N,u,D)\geq1 $ such that for all ball $ B(x,r) $ relatively compact in $ D $,
$$ u_{M}(x) \leq \frac{K}{\nu_{N}r^{N}}\int_{B(x,r)}u_{M}(\xi)d\lambda(\xi),$$
for all $ M\geq0. $ Here,  $ u_{M}:=\max\lbrace u,-M\rbrace+M. $
\end{definition}
This and the following definition are generalizations of subharmonic function  given by  J. Riihentaus (see \cite{Rii2}).

\begin{definition}
A function $ u:D\rightarrow [-\infty,+\infty) $ is called $ K- $quasi-nearly subharmonic n.s. (in the narrow sense), if $ u $ is (Lebesgue) mearsurable, its positive part $ u^{+} $ is locally integrable and there exists a constant $ K=K(N,u,D)\geq1 $ such that for all ball $ B(x,r) $ relatively compact in $ D $,
$$ u(x) \leq \frac{K}{\nu_{N}r^{N}}\int_{B(x,r)}u(\xi)d\lambda(\xi).$$

\end{definition}

\begin{theorem}[Riihentaus] 
 Let   $u$ be a $K$-quasi-nearly subharmonic function n.s. on a domain $D$ of 
$\mathbb{R^{N}}$  $(N\geq2)$.
\begin{itemize}
 \item[(i)] \textit{ If $u\not\equiv-\infty$, then $u$ is finite almost everywhere  and is locally integrable on $D$;}
 \item[(ii)]\textit{ The function $u$ is locally bounded above on $D$.}
\end{itemize}
\end{theorem}
See \cite[Proposition 1]{Rii2}

 \section{Main Results and their proofs}
 
 \begin{theorem}
If $u$ is $K$-quasi nearly subharmonic, then so is  $u^{*}$.
\end{theorem}

All we need to show is that $(u^{*})_{M}$ satisfies the quasi mean inequality. We start by proving  that 
\begin{equation}
 (u^{*})_{M}=(u_{M})^{*}. \label{5173} 
 \end{equation}

\begin{lemma}
Let $(a_{n})$ be a convergent sequence of real numbers and $k$ a constant. Then
$$\lim_{n\rightarrow+\infty}\max\lbrace a_{n},k\rbrace=\max\lbrace \lim_{n\rightarrow+\infty}a_{n},k\rbrace.$$
\end{lemma}
\begin{proof}
 Let  $f(x):=\max\lbrace x,k\rbrace$. This is a continuous function. The left side of the above equation equals $\lim_{n\rightarrow+\infty}f(a_{n})$ and the right side equals $f(\lim_{n\rightarrow+\infty}a_{n})$. They are equal by continuity of $f$.
\end{proof}

Now we can prove (\ref{5173}). Let $(r_{n})$ be a sequence of real numbers that approaches 0, as $n\rightarrow+\infty$ and let $\overline{B(\zeta,r_{n})}\subset D$. It is easy to check that
\begin{align*}
\sup_{x}\left\lbrace \max\lbrace u(x),-M\rbrace+M\right\rbrace&= \sup_{x}\left\lbrace \max\lbrace u(x),-M\rbrace\right\rbrace+M\\
&=\max\lbrace\sup_{x} u(x),-M\rbrace+M,
\end{align*}
 where  the suprema are taken over the ball $B(\zeta,r_{n})$. By letting $n\rightarrow+\infty$, the left side of the first equation approaches $(u_{M})^{*}(\zeta)$, and the last expression is equal, according to Lemma 2.3, to 
$$\max\lbrace\lim_{n\rightarrow+\infty}\sup_{x}u(x),-M\rbrace+M,$$ 
which is $(u^{*})_{M}(\zeta)$. This proves (\ref{5173}).

\begin{proof}[Proof of  Theorem 2.1] Take $\overline{B(x,r)}\subset D.$ We have 
\begin{align*}
u_{M}(x)&\leq \frac{K}{\nu_{N}r^{N}}\int_{B(x,r)}u_{M}(\xi)d\lambda(\xi)\\
&\leq\frac{K}{\nu_{N}r^{N}}\int_{B(x,r)}(u_{M})^{*}(\xi)d\lambda(\xi)\\
&=\frac{K}{\nu_{N}r^{N}}\int_{B(x,r)}(u^{*})_{M}(\xi)d\lambda(\xi),\\
\end{align*}
by (\ref{5173}). We notice that the last integral is a continuous function of $x$, since the integrand is integrable; it majorizes $u_{M}$ thus it also majorizes $(u_{M})^{*}$, which equals $(u^{*})_{M}$, by (\ref{5173}). We obtain
$$(u^{*})_{M}\leq\frac{K}{\nu_{N}r^{N}}\int_{B(x,r)}(u^{*})_{M}(\zeta)d\lambda(\zeta),$$ as required.
 \end{proof}

 \begin{theorem}
 Let $u$ be a $K$-quasi-nearly subharmonic function n.s.  on $D$, and  
 $$N:=\lbrace x\in D: u(x)<0\rbrace$$
 be the negative set of $ u $. If the interior of $ N $ is not empty, then $ u $ is nearly subharmonic. 
\end{theorem}

\begin{proof}
We need to prove that $ u $  satisfies the mean value inequality everywhere on $D$. Take a ball $B(x,r)$ relatively compact in $N$. We have 
$$u(x)\leq K\left( \frac{1}{\nu_{Nr^{N}}}\int_{B(x,r)}ud\lambda\right) \leq0. $$
We know that almost every $x$ in the interior of $N$ is a Lebesgue point, meaning that  the above normalized integral within parenthesis converges to $u(x)$, as $r\rightarrow0^{+}$. For such an $x$ and by letting $r\rightarrow0$ we get $u(x)\leq Ku(x)\leq0$  and thus $K\leq1$. Since by definition $K\geq1$, we obtain $K=1$. Now, by Theorem 1.6 (i) $u$ is locally integrable and satisfies the mean value inequality. Thus $u$ is nearly subharmonic.

\end{proof}

 \begin{corollary}
 Let $u$ be a $K$-quasi-nearly subharmonic function n.s.  on $D$. Then, either $u^{*}$ is subharmonic on $D$, or $u^{*}\geq0$  and  is $K$-quasi-nearly subharmoic n.s. on $D$.
\end{corollary}
\begin{proof}
If the negative set of $ u $ has interior points, then by Theorem 2.6   $u^{*}$ is subharmonic on $D$. Next, assume that $N$ has empty interior and let us prove that 
\begin{equation}
u^{*}(x)\geq 0
\end{equation}
for all $x\in D$. If $x\in D\setminus N$ or if $ N $ is empty, there is nothing to prove. Let $x\in N$. There exits a sequence $\lbrace x_{n}\rbrace\subset D\setminus N$ converging to $x$, as $n\rightarrow+\infty$. We have
$$0\leq \limsup_{n\rightarrow+\infty}u(x_{n})\leq\limsup_{n\rightarrow+\infty}u^{*}(x_{n}) \leq u^{*}(x),$$ since $u^{*}$ is by construction upper semi-continuous.

 To prove that if $ u $ is $ K- $quasi-nearly subharmonic n.s., then so is $ u^{*} $,  we follow  J. Riihentaus \cite[pg 5-6]{Rii3} and make maybe some minor adjustments. First notice that  $u$ is locally bounded above, according to Theorem 2.1., and so  $u^{*}$ is well-defined. Next, being upper semi-continuous, it is also  measurable and integrable. Thus we just need to prove that $u^{*}$ satisfies the quasi-mean inequality everywhere. Let $B(\zeta,\rho)$ be an arbitrary ball relatively compact in $D$. There exists $0>\delta$ such that $B(\zeta,\rho+2\delta)$ is steel relatively compact in $D$. We have  
 
$$u(x)\leq \frac{K}{\nu_{N}\rho^{N}}\int_{B(x,\rho)}u(\xi)d\lambda(\xi),$$
for all $x\in B(\zeta,\delta)$. 
By taking the limit superior, we obtain
$$\limsup_{x\rightarrow\zeta}u(x)\leq K \limsup_{x\rightarrow\zeta}\left(\frac{1}{\nu_{N}\rho^{N}}\int_{B(x,\rho)}u(\xi)d\lambda(\xi) \right).$$ 
Let $\phi(x)$ designate the function defined by the integral within parenthesis. Since $u$ is integrable, the function $\phi$ is  continuous in $B(\zeta,\delta)$, according to a classic theorem of measure theory.  Thus the limit of the right side is in fact  $\phi(\zeta)$. The left side limit is the upper semi-continuous regularization of $u$. We thus get
$$u^{*}(\zeta)\leq \frac{K}{\nu_{N}\rho^{N}}\int_{B(\zeta,\rho)}u(\xi)d\lambda(\xi)\leq \frac{K}{\nu_{N}\rho^{N}}\int_{B(\zeta,\rho)}u^{*}(\xi)d\lambda(\xi).$$  

\end{proof}

\end{document}